\input ./style/arxiv-general.cfg
\documentclass[aop,MSNbibl,seceqn,dvips]{arximspdf}
\makeatletter
   \@ifpackageloaded{graphicx}{}{\usepackage{graphicx}}
\makeatother

\doi{10.1214/15-AOP1049}
\volume{44}
\issue{5}
\pubyear{2016}
\firstpage{3335}
\lastpage{3356}
\docsubty{FLA}

\makeatletter
\newtheorem{theorem}{Theorem}[section]
\newtheorem{corollary}{Corollary}[section]
\newproclaim{example}{Example}[section]
\newtheorem{lemma}{Lemma}[section]
\newtheorem{proposition}{Proposition}[section]

\newproclaim{remark}{Remark}

\newproclaim{definition}{Definition}
\newproclaim{exercise}{Exercise}

\newcommand{\cFE}{c_{{\mathsf{FE}}}} 
\newcommand{\cP}{c_{\mathsf{P}}}
\newcommand{\cBK}{c_{\mathsf{BK}}}
\newcommand{\cRSW}{c_{\mathsf{RSW}}}
\newcommand{\cN}{c_{\mathsf{N}}}
\newcommand{\sfF}{\mathsf{F}}

\newcommand{\calC}{\mathcal{C}}
\newcommand{\calD}{\mathcal{D}}
\newcommand{\calE}{\mathcal{E}}
\newcommand{\calH}{\mathcal{H}}
\newcommand{\calR}{\mathcal{R}}

\newcommand{\bbE}{\mathbb{E}}
\newcommand{\bbL}{\mathbb{L}}
\newcommand{\bbP}{\mathbb{P}}
\newcommand{\bbR}{\mathbb{R}}
\newcommand{\bbT}{\mathbb{T}}
\newcommand{\bbV}{\mathbb{V}}
\newcommand{\bbZ}{\mathbb{Z}}

\newcommand{\eqref}[1]{(\ref{#1})}

\newcommand{\Var}{\operatorname{\bbV\mathrm{ar}}}

\newcommand{\1}{\mathbf{1}}

\newcommand{\ep}{\varepsilon}

\makeatother

\begin{document}
\begin{frontmatter}

\title{A quantitative Burton--Keane estimate under strong FKG condition}
\runtitle{A quantitative Burton--Keane estimate}

\begin{aug}
\author[A]{\fnms{Hugo}~\snm{Duminil-Copin}\thanksref{m1,T1}\ead
[label=hugo]{hugo.duminil@unige.ch}},
\author[B]{\fnms{Dmitry}~\snm{Ioffe}\thanksref{m2,T2}\ead
[label=dima]{ieioffe@ie.technion.ac.il}}
\and
\author[A]{\fnms{Yvan}~\snm{Velenik}\corref{}\thanksref{m1,T1}\ead
[label=yvan]{yvan.velenik@unige.ch}}
\runauthor{H. Duminil-Copin, D. Ioffe and Y. Velenik}
\thankstext{T1}{Supported by the NCCR SwissMap founded by the Swiss NSF.}
\thankstext{T2}{Supported by Technion research Grant 2020225.}
\affiliation{University of Geneva\thanksmark{m1} and
Technion\thanksmark{m2}}

\address[A]{H. Duminil-Copin\\
Y. Velenik\\
Section de Math\'ematiques\\
Universit\'e de Gen\`eve\\
1211 Gen\`eve 4\\
Switzerland\\
\printead{hugo}\\
\phantom{E-mail:\ }\printead*{yvan}}
\address[B]{D. Ioffe\\
Faculty of IE\&M\\
Technion\\
Haifa 32000\\
Israel\\
\printead{dima}}
\end{aug}

%
\received{\smonth{10} \syear{2014}}
%
\revised{\smonth{6} \syear{2015}}

%
\begin{abstract}
We consider translationally-invariant percolation models on $\mathbb Z^d$
satisfying the finite energy and the FKG properties. We provide
explicit upper
bounds on the probability of having two distinct clusters going from the
endpoints of an edge to distance $n$ (this corresponds to a finite size version
of the celebrated Burton--Keane
[\textit{Comm. Math. Phys.} \textbf{121} (1989) 501--505]
argument proving
uniqueness of
the infinite-cluster). The proof is based on the generalization of a reverse
Poincar\'e inequality proved in Chatterjee and Sen (2013).
As a consequence, we obtain
upper bounds on the probability of the so-called four-arm event for planar
random-cluster models with cluster-weight $q\ge1$.
\end{abstract}

%

%
\end{frontmatter}

\section{Introduction and main result}
This article is devoted to deriving a weak reverse Poincar\'e-type inequality
for percolation models satisfying strong association
and finite-energy properties, and examining some of its consequences. Let
$\Lambda$ be a finite set and consider a
percolation model on $\Lambda$, that is, a random binary field
$\omega\in\{0,1\}^\Lambda$. The value of the field at $i\in\Lambda
$ is denoted
by $\omega_i$, and the field on the complementary set $\Lambda
\setminus i$ is
denoted by $\omega^i$. The law of $\omega$ on $\{0,1\}^\Lambda$ is
denoted by~$\bbP$. There is a standard partial order $\prec$ on $\{0,1\}^\Lambda
$, and a
function $f$ on $\{0,1\}^\Lambda$ is said to be nondecreasing if $f
(\omega)
\leq
f (\psi)$ whenever $\omega\prec\psi$. An event $A\subset\{0,1\}
^\Lambda$ is
said to be nondecreasing if its indicator function $\1_A$ is.

We will be interested in percolation models satisfying the following two
conditions:

\begin{longlist}[($\mathsf{FKG}$)]
\item[($\mathsf{FE}$)] \emph{Finite energy}: There exists
$\cFE>0$ such that, for any $i\in\Lambda$ and
$\omega\in\{0,1\}^\Lambda$,
%
\begin{equation}
\label{eq:FE} \bbP(\omega_i =1|\omega_j,j\ne i)\in(\cFE,
1-\cFE).
\end{equation}

\item[($\mathsf{FKG}$)] \emph{Strong positive association}:
For any $i\in\Lambda$ and $\xi\prec\psi$ in $\{0,1\}^{\Lambda
\setminus\{i\}}$,
%
\begin{equation}
\label{eq:strongFKG} \bbP(\omega_i =1|\omega_j=
\xi_j,j\ne i)\leq\bbP( \omega_i =1|\omega_j=
\psi_j,j\ne i).
\end{equation}
\end{longlist}

Recall that \eqref{eq:strongFKG} is equivalent to the so-called FKG lattice
condition; see \cite{Gri06}, Theorem (2.24). For a discussion of the
relation between this condition and the weaker condition of positive
association (characterized by the FKG inequality), we refer the reader
to \cite{ACCN}.

Before stating the theorem, let us define two more objects. For a
configuration $\omega$ and $i\in\Lambda$, define the configurations
$\omega^i\times1$ and $\omega^i\times0$ obtained from $\omega$ by
changing the
state $i$ to $1$ and $0$, respectively. Define
\[
\nabla_i f (\omega) \stackrel{\mathrm{def}} {=}f \bigl(
\omega^i\times1 \bigr)- f \bigl(\omega^i\times 0 \bigr).
\]
Also,
for a nondecreasing event $A$,
define the set $\operatorname{Piv}_i(A)$ of configurations $\omega$ such that
$\omega^i\times1\in A$ and $\omega^i\times0\notin A$. Equivalently,
$\operatorname{Piv}_i(A) \stackrel{\mathrm{def}}{=} \{\omega:
\nabla_i \1_A (\omega)=1  \}$.

\begin{theorem}
\label{thm:BK-Var}
Consider a percolation model on a finite set $\Lambda$ satisfying
$(\mathsf{FE})$ and $(\mathsf{FKG})$. Then there exists $\cP=
\cP(\cFE)>0$ such that, for any nondecreasing function
$f:\{0,1\}^\Lambda\longrightarrow\bbR$,
%
\begin{equation}
\label{eq:BK-Var} \Var \bigl( f (\omega) \bigr) \geq \cP %
\sum
_{i\in\Lambda} \bigl( \bbE[\nabla_i f]
\bigr)^2.
\end{equation}
In particular, for any nondecreasing event $A$
%
\begin{equation}
\label{eq:main} \bbP(A) \bigl(1-\bbP(A) \bigr) \geq\cP \sum
_{i\in\Lambda} \bbP \bigl(\operatorname{Piv}_i(A)
\bigr)^2.
\end{equation}
\end{theorem}

We emphasize that the constant $\cP$ is not depending on the size of
$\Lambda$.
One may think of this theorem as a weak reverse Poincar\'e inequality. Indeed,
when $ \{\omega_i: i\in\Lambda  \}$ are independent, the
standard discrete
Poincar\'e inequality (see, e.g., \cite{GS12}) states that
%
\begin{equation}
\label{eq:directP} \bbP(A) \bigl(1-\bbP(A) \bigr) \leq\frac{1}{4}\sum
_{i\in\Lambda} \bbP \bigl(\operatorname{Piv}_i(A) \bigr).
\end{equation}
In the independent case, $\bbP (\operatorname{Piv}_i(A) ) =
\bbP(A|\omega_i=1)-\bbP(A|\omega_i=0) \stackrel{\operatorname
{def}}{=}I_A (i)$ is the so-called
\emph{influence} of $i$ on $A$. Let us also mention that some
inequalities for
influences in models with dependency have been obtained by encoding strongly
positively-associated measures in terms of the Lebesgue measure on the
hypercube
$[0,1]^\Lambda$. Nevertheless, these inequalities bound influences
from below;
see \cite{Gri06}, Theorem (2.28) and \cite{GraGri11}. They are therefore
not directly relevant here.

Inequality \eqref{eq:BK-Var} was derived in the independent Bernoulli
case in \cite{ChS13}. The latter work was one of the motivations for
our study
of dependent models here.

Our proof of \eqref{eq:BK-Var} hinges on the following simple but apparently
new observation, which may be of independent interest. Fix $0<p<1$;
given a realization $\omega\in\{0,1\}^\Lambda$ of the percolation
model, we
construct a field $\sigma\in\{0,1\}^\Lambda$ of (conditionally on
$\omega$)
independent random variables, $\sigma_i$ taking value $1$ with probability
\[
\widehat{\bbP}(\sigma_i = 1 |\omega)= \frac{p\cdot\omega_i}{\bbP(\omega_i=1|\omega^i)},
\]
for each $i\in\Lambda$. Then the distribution of the field $\sigma$, once
integrated over $\omega$, enjoys a form of \emph{negative}
dependence. Namely,
for any $i\in\Lambda$ and for any nondecreasing functions
$f:\{0,1\}^{\Lambda\setminus\{i\}}\rightarrow\bbR$ and
$g:\{0,1\}\rightarrow\bbR$,
\[
\widehat{\bbE} \bigl[ f \bigl(\sigma^i \bigr)g(\sigma_i)
\bigr] \leq \widehat{\bbE} \bigl[f\bigl(\sigma^i\bigr) \bigr]
\widehat{\bbE} \bigl[g(\sigma _i) \bigr].
\]
This is proved in Theorem~\ref{thm:uZ} below, together with additional relevant
properties.

\section{Applications}
\subsection{Some examples of percolation models}
Our applications to percolation models will be mostly dealing with connectivity
properties of the graph induced by $\{i\in\Lambda:\omega_i=1\}$. For
simplicity, we will focus on \emph{bond percolation} models---similar results
would also hold for so-called \emph{site percolation} models. The set
$\Lambda$
is now the edge-set $E_G$ of a finite graph $G=(V_G,E_G)$. The edge $i$
is said
to be \emph{open} (resp., \emph{closed}) if $\omega_i=1$ (resp.,
$\omega_i=0$).
The
configuration $\omega$ can therefore be seen as a subgraph of $G$ with vertex
set $V_G$ and edge set composed of open edges. Two vertices $x$ and $y$ are
said to be \emph{connected} if they belong to the same connected
component of
$\omega$ (we denote the event that $x$ and $y$ are connected by
$x\longleftrightarrow y$). Connected components of $\omega$ are called
\emph{clusters}.

The most classical example of a percolation model is provided by Bernoulli
percolation. This model was introduced by Broadbent and Hammersley in
the 1950s~\cite{BroHam57}. In this model, each edge $i$ is open with
probability
$p$, and closed with probability $1-p$ independently of the states of
the other
edges. For general background on Bernoulli percolation, we refer the
reader to
the books \cite{Gri99,Kes82}.

More generally, the states of edges may not be independent. In such
case, we
speak of a \emph{dependent} percolation model. Among classical
examples, we
mention the \emph{random-cluster model} (or \emph{Fortuin--Kasteleyn
percolation})
introduced by Fortuin and Kasteleyn in \cite{ForKas72}. Let $o(\omega
)$ be the
number of open edges in $\omega$, $c(\omega)$ be the number of closed
edges and
$k(\omega)$ be the number of clusters. The probability measure $\phi_{p,q,G}$
of the random-cluster model on a finite graph $G$ with parameters $p\in[0,1]$
and
$q>0$ is defined by
\[
\phi_{p,q,G} \bigl( \{\omega \}\bigr) \stackrel{\mathrm{def}} {=}
\frac{p^{o(\omega)}(1-p)^{c(\omega)}q^{k(\omega)}}{Z_{p,q,G}}
\]
for every configuration $\omega$ on $G$, where $Z_{p,q,G}$ is a normalizing
constant referred to as the \emph{partition function}.

The random-cluster models satisfy ($\mathsf{FE}$) for $q>0$ and
($\mathsf{FKG}$) for any $q\ge1$. For this reason, random-cluster
models are
good examples of models satisfying our two assumptions, but they are
not the
only ones. The uniform spanning tree ($\bbP$ is simply the uniform
measure on
trees containing every vertices of $G$) is a typical example of a model not
satisfying ($\mathsf{FE}$).

Our applications provide upper bounds on the probability of having
two distinct clusters from the inner to the outer boundaries of annuli.
In two dimensions because of dual connections, the
usual name would be four-arm type events, namely probabilities of
having two
long disjoint clusters attached to two vertices of a given edge.

In
order to deduce such estimates for individual bonds from \eqref{eq:BK-Var}
or \eqref{eq:main}, we need to assume some form of translation
invariance.

\subsection{First application}
To give a simple illustration of how Theorem~\ref{thm:BK-Var} might be
put to
work, let us mention the following result. Consider the $d$-dimensional torus
$\bbT_n^{(d)}$ of size $2n+1$ and denote by $\tilde A_2^e(n)$ the event
that
the edge $e$ is pivotal for the existence of an open circuit of nontrivial
homotopy in $\bbT_n^{(d)}$.

\begin{proposition}\label{prop:4 arm torus}
Let $d\ge2$, there exists
$c_{\tilde{\mathsf{A}}_2}=c_{\tilde{\mathsf{A}}_2}(\cFE,d)>0$ such
that, for
every $n\ge1$ and any edge $e$ of $\bbT_n^{(d)}$,
\[
\bbP\bigl[\tilde A_2^e(n)\bigr]\le\frac{c_{\tilde{\mathsf{A}}_2}}{n^{d/2}},
\]
where $\bbP$ is the law of an arbitrary translation invariant
percolation model
on $\bbT_n^{(d)}$ satisfying $(\mathsf{FE})$ and $(\mathsf{FKG})$.
\end{proposition}

Note that $\tilde A_2^e(n)$ is basically the event that there are two disjoint
clusters emanating from the end-points of $e$ and going to distance
$n$, with
some additional topological requirement on the macroscopic structure of these
clusters (among these requirements, they should join into a cluster of $
\bbT_n^{(d)}$). This additional condition is not so nice, and it does not
directly apply to models on $\bbZ^d$. We would like to replace this by the
event that there are two disjoint clusters going from the end-points of
\emph{some fixed bond} $e$ to distance $n$.
Let $A_2^e(n)$ be the event that there are two disjoint clusters going from
the endpoints of the edge $e$ to distance $n$.

In the next two applications, we explain two ways of deriving upper
bounds on
$\bbP(A_2^e(n))$.

\subsection{A quantitative Burton--Keane argument}
\label{sub:QBK}
Our second application is an extension of the results of \cite{ChS13} to
arbitrary bond percolation models $\bbP$ on $\bbZ^d$ which satisfy
$(\mathsf{FE})$, $(\mathsf{FKG})$ and are invariant under translations:
\begin{longlist}[$(\mathsf{TI})$]
\item[$(\mathsf{TI})$] The measure $\bbP$ is invariant under shift
$\tau_x:\{0,1\}^{\bbZ^d}\rightarrow\{0,1\}^{\bbZ^d}$ defined by
\[
\tau_x(\omega)_{(u,v)}\stackrel{\mathrm{def}} {=}\omega
_{(u+x,v+x)}\qquad \forall u,v\in \bbZ^d.
\]
\end{longlist}

\begin{theorem}\label{prop:quantitative BK}
Consider a percolation model on $\bbZ^d$ satisfying $(\mathsf{FE})$,
$(\mathsf{FKG})$ and $(\mathsf{TI})$. Then there exists $\cBK>0$
such that, for any edge $e$,
%
\begin{equation}
\label{eq:QBK} \bbP\bigl[A_2^e(n)\bigr]\le
\frac{\cBK}{(\log n)^{d/2}}.
\end{equation}
\end{theorem}

As we have already mentioned a quantitative Burton--Keane argument leading
to \eqref{eq:QBK} for Bernoulli percolation-type models was developed
in \cite{ChS13}.
In the case of Bernoulli site percolation, polynomial order upper bounds
on $\bbP[A_2^e(n)]$ were derived in the recent paper~\cite{Cerf13}
via a
clever refinement of techniques introduced by~\cite{AizKesNew87} and
\cite{GanGriRus88}.

\subsection{Continuity of percolation probabilities away from critical points}
Consider a one-parametric family $\{\bbP_\alpha\}_{\alpha\in
(a,b)}$ of
bond or site strong-$\mathsf{FKG}$ percolation
models on $\bbZ^d$. Define percolation probabilities
%
\begin{equation}
\label{eq:theta-alpha} \theta(\alpha) \stackrel{\mathrm{def}} {=}\bbP_\alpha(0
\leftrightarrow\infty).
\end{equation}
Assume that the measures $\bbP_\alpha$ satisfy the finite energy condition
($\mathsf{FE}$) uniformly
over compact intervals of $(a, b )$, and assume that $\theta>0$ on $(a,b)$.
At last, assume that $\alpha\mapsto\bbP_\alpha$ is increasing (in the
$\mathsf{FKG}$-sense), that
is, assume that $\bbP_\alpha$ is stochastically dominated by $\bbP
_\beta$
whenever $a<\alpha\leq\beta<b$.
We shall say that $\alpha\mapsto\bbP_\alpha$ is continuous at
$\alpha_0\in
(a,b)$ if the map $\alpha\mapsto\bbP_\alpha(f )$ is continuous at
$\alpha_0$ for any local function $f$.

\begin{theorem}
\label{thm:cont}
Under the above conditions: $\alpha\mapsto\theta(\alpha)$ cannot have
jumps at continuity points of $\alpha\mapsto\bbP_\alpha$.
\end{theorem}

In the case of Bernoulli percolation, continuity comes for free and
Theorem~\ref{thm:cont} implies continuity of percolation probabilities away
from critical points, as it was originally proved in \cite
{AizKesNew87}. In
the case of FK-percolation for the Ising model on $\bbZ^d$, proving continuity
of measures seems to be on the same level of difficulty as proving continuity
of percolation probabilities \cite{Bod2006}. On the other hand, in view
of \cite{Bod2006} and \cite{AD-CS2014}, Theorem~\ref{thm:cont} does
imply continuity of the site $+$-spin percolation away from critical inverse
temperature for the latter.

\subsection{Spanning clusters and polynomial decay}
Proposition~\ref{prop:4 arm torus} and Theorem~\ref{prop:quantitative
BK} are
based on \eqref{eq:main}. Yet, \eqref{eq:BK-Var} provides us with additional
degrees of freedom: one can try various model-dependent monotone
functions $f$.

Let $\bbP$ be a bond percolation measure on $\bbZ^d$. Consider the boxes
$\Lambda_k \stackrel{\mathrm{def}}{=}[-k,\ldots,k]^d$ and the
annuli $A_{m,n} \stackrel{\mathrm{def}}{=}
\Lambda_n\setminus
\Lambda_m$ for $0 <m <n$. Let $N= N_{m,n}$ be the number of distinct clusters
of $\partial\Lambda_n$ \emph{crossing} $A_{m,n}$ in the \emph
{restriction} of
the
percolation configuration to the bonds of $\Lambda_n$. In the sequel,
we shall
use $\eta$ for a percolation configuration on $\bbZ^d\setminus
\Lambda_m$,
$\omega$ for a percolation configuration on $\Lambda_m$, and $\eta
\times\omega$
for the configuration
obtained by merging the two previous configurations. For a given $\eta
$, the
function $\omega\mapsto N (\eta\times\omega)\stackrel{\mathrm
{def}}{=}N^\eta(\omega
)$ is
decreasing. Hence, \eqref{eq:BK-Var} implies that
%
\begin{equation}
\label{eq:N-bound} \Var\bigl(N^\eta(\omega) | \eta\bigr) \geq\cP \sum
_{e\in\calE_{\Lambda_m}} \bigl(\bbE\bigl(\nabla_e
N^\eta(\omega) |\eta\bigr)\bigr)^2.
\end{equation}
Above $\calE_{\Lambda_m}$ is the set of nearest neighbor bonds of
$\Lambda_m$.
Note that, for any $e\in\calE_{\Lambda_m}$,
\[
- \nabla_e N^\eta_{m,n} (\omega) \geq
\1_{A_2^e (2n )} (\eta\times \omega).
\]
Therefore, we infer from \eqref{eq:N-bound} the following corollary.

\begin{theorem}
\label{cor:N-bound-cor}
Consider a percolation model on $\bbZ^d$ satisfying $(\mathsf{FE})$,
$(\mathsf{FKG})$ and $(\mathsf{TI})$. Then, for any edge $e$,
%
\begin{eqnarray}
\label{eq:N-bound-cor} \bbP\bigl(A_2^e (2n) \bigr)&\leq& \sqrt
{ \frac{1}{\cP(2m)^d} \bbE\bigl( \Var\bigl(N^\eta_{m,n}
(\omega) | \eta\bigr)\bigr)}
\nonumber
\\[-8pt]
\\[-8pt]
\nonumber
&\leq& \sqrt{ \frac{1}{\cP(2m)^d} \Var(N_{m,n})},
\end{eqnarray}
for any $0 < m< n$.
\end{theorem}

Of course, \eqref{eq:N-bound-cor} is useful only when one is able to control
the number of crossing clusters of $A_{m.n}$, specifically $\bbE(
\Var(N^\eta_{m,n} (\omega)  | \eta))$. This requires
work: a
trivial upper bound of order $m^{2(d-1)}$ gives nothing even in two dimensions.
Settling this in any dimension would be a feat even in the case of Bernoulli
percolation; see \cite{Aiz97}. For the moment, it is not clear to us
that a
nice closed form bound can be obtained in the full generality suggested by
$(\mathsf{FE})$, $(\mathsf{FKG})$ and $(\mathsf{TI})$, even if one requires
ergodicity instead of just translation invariance.

In the case of Bernoulli site percolation, the following bound was derived
using very different methods based on independence (see \cite{Cerf13}):
%
\begin{equation}
\label{eq:Cerf} \bbP\bigl(A_2^e (2n) \bigr)\leq
\frac{c\log n}{ n^{d/2}} \bbE( \sqrt{N_{n,
2n}}).
\end{equation}
Unlike \eqref{eq:N-bound-cor}, \eqref{eq:Cerf} always gives a nontrivial
polynomial decay, even
if the roughest possible bound $N_{n, 2n} \leq cN^{d-1}$ is used.

\subsection{Four-arm event for critical planar random-cluster models with
\texorpdfstring{$q\ge1$}{q>=1}}
Using very recent results of \cite{Dum13,DumSidTas13} for the
random-cluster model on $\bbZ^2$, the distribution of the number of crossing
clusters can be controlled, and the upper bound \eqref{eq:N-bound-cor} implies
the following refinement of Theorem~\ref{prop:quantitative BK}, which
is of the
same order as the bound of Proposition~\ref{prop:4 arm torus}.

\begin{theorem}\label{prop:4 arm}
Let $d=2$, $q\in[1,4]$, there exists $c_{\mathsf A_2}=c_{\mathsf A_2}(p,q)>0$
such that, for any edge $e$ and every $n\ge1$,
\[
\phi_{p,q,\bbZ^2}\bigl[A_2^e(n)\bigr]\le
\frac{c_{\mathsf A_2}}{n},
\]
where $\phi_{p,q,\bbZ^2}$ is the unique infinite volume
random-cluster measure with edge-weight $p$ and cluster-weight $q$ (see
Section~\ref{sec:RSW} for a precise definition).
\end{theorem}

The proof is easy whenever $p\neq p_c$.
In the critical case $p=p_c$, the proof is based on
Russo--Seymour--Welsh (RSW)
bounds obtained in \cite{Dum13,DumSidTas13}. We give two arguments:
one is
based on the implied mixing properties of $\phi_{p,q,\bbZ^2}$ and on a
subsequent reduction to Proposition~\ref{prop:4 arm torus}. The second
directly relies on RSW bounds to check that one can fix $\varepsilon>0$
such that
$\{\phi_{p,q,\bbZ^2} (N_{\varepsilon n, n}^2)\}$ is a
bounded sequence.
Then \eqref{eq:N-bound} applies.

Note that the phase transition is expected to be discontinuous for
$q>4$ (see
the discussion in Section~\ref{sec:RSW}) and the probability of
$A_2(n)$ should
decay exponentially fast at every $p$.

\section{Proof of Theorem \texorpdfstring{\protect\ref{thm:BK-Var}}{1.1}}\label{sec:2}
We shall prove Theorem~\ref{thm:BK-Var} with $\cP= \frac{\cFE{}^3}{
(2-\cFE)^2}$.

From now on in this section, we fix a finite set $\Lambda$. Consider a
percolation model on $\Omega\stackrel{\mathrm{def}}{=}\{0,1\}
^\Lambda$ satisfying
$(\mathsf{FE})$
and
$(\mathsf{FKG})$ and let $\bbP$ be the law of the random configuration
$\omega$.
Furthermore, for $I\subset\Lambda$, we define
$\omega_{I}\stackrel{\mathrm{def}}{=} \{\omega_i: i\in I
 \}$
and $\omega^{I}\stackrel{\mathrm{def}}{=} \{\omega_i: i\notin
I  \}$. To keep notation
compatible, we set $\omega_i$ and $\omega^i$ when $I=\{i\}$.

Recall that $\omega^i\times1$ and $\omega^i\times0$ denote the
configurations
obtained from $\omega$ by setting the value of $\omega_i$ to 1 and 0,
respectively.\vspace*{6pt}

\begin{center}
\emph{To lighten the notation, we write $p=\cFE/2$ for the rest of
this section.}
\end{center}

\subsection{A representation of fields satisfying $(\mathsf{FE})$ and
$(\mathsf{FKG})$}
In order to prove Theorem~\ref{thm:BK-Var}, we introduce an auxiliary
Bernoulli field $\sigma\in\{0,1\}^\Lambda$ and utilize the projection
method of \cite{ChS13}
with respect to $\sigma$-algebras generated by this auxiliary field. The
efficiency of
such approach hinges on the fact that $\sigma_i$'s happen to be
negatively correlated in the sense specified in $\mathsf{P3}$ of
Theorem~\ref{thm:uZ} below.

\begin{definition}
Consider a probability space $(\widehat\Omega,\widehat\bbP)$ containing
$(\Omega,\bbP)$ and an additional field $\sigma\in\{0,1\}^\Lambda$ which,
conditionally on $\omega\in\Omega$, has independent entries
satisfying, for
every $i\in\Lambda$,
%
\begin{equation}
\label{eq:defZ} \widehat\bbP(\sigma_i=1|\omega)=\frac{p\cdot\omega_i}{\bbP
(\omega_i=1|\omega^i)} .
\end{equation}
\end{definition}

Note that by our choice of $p$, which is adjusted to the finite energy
property~\eqref{eq:FE}, the
right-hand side of \eqref{eq:defZ} always belongs to $[0,1]$.

We claim that $\sigma$ enjoys the following set of properties.

\begin{theorem}
\label{thm:uZ}
Let $\omega$ be a field satisfying $(\mathsf{FE})$ and $(\mathsf
{FKG})$. Then:
\begin{longlist}[$\mathsf{P1}$]
\item[$\mathsf{P1}$] For each $i\in\Lambda$, if $\sigma_i=1$, then
$\omega_i=1$.
\item[$\mathsf{P2}$] For any $i\in\Lambda$ and any nondecreasing function
$f:\Omega\rightarrow\bbR$, the conditional expectation $\widehat
{\bbE}(f
(\omega
)|\sigma_i)$ is also nondecreasing.
\item[$\mathsf{P3}$] For any $i\in\Lambda$ and for any nondecreasing
functions
$f:\{0,1\}^{\Lambda\setminus\{i\}}\rightarrow\bbR$ and
$g:\{0,1\}\rightarrow\bbR$,
%
\begin{equation}
\label{eq:NegAs} \widehat{\bbE} \bigl[ f \bigl(\sigma^i \bigr)g(
\sigma_i) \bigr] \leq \widehat{\bbE} \bigl[f\bigl(\sigma^i
\bigr) \bigr]\widehat{\bbE} \bigl[g(\sigma _i) \bigr].
\end{equation}
\item[$\mathsf{P4}$] The family $\{\sigma_i-p:i\in\Lambda\}$ is
free in
$\bbL^2(\widehat\Omega,\widehat\bbP)$.
\end{longlist}
\end{theorem}

Property $\mathsf{P3}$ provides a form of negative association.
It is weaker than the usual form of negative association [which
corresponds to the analogue of \eqref{eq:NegAs} with $i$ and
$\Lambda\setminus\{i\}$ replaced by arbitrary disjoint subsets
$A,B\subset\Lambda$], but stronger than other related notions, such
as totally
negative dependence (see \cite{Daly} for this and other forms of negative
dependence).

\begin{pf*}{Proof of Theorem \ref{thm:uZ}} \emph{Property} $\mathsf{P1}$.
The first property follows directly from the definition of $\sigma$.

\emph{Property $\mathsf{P2}$.} Let us first prove that, for each
$i\in\Lambda$, $\sigma_i$ is a Bernoulli random variable of
parameter $p$,
independent of $\omega^i$.
This follows from \eqref{eq:defZ} and the following computation:
\[
\widehat{\bbP}\bigl(\sigma_i = 1;\omega^i\bigr) =
\widehat{\bbP}\bigl(\sigma_i = 1;\omega_i =1;
\omega^i\bigr) = \frac{p}{\bbP({\omega_i =1}|\omega^i)} \bbP\bigl( \omega_i
=1; \omega^i\bigr)= p \bbP\bigl(\omega^i\bigr).
\]
Hence, $\sigma_i$ is indeed a Bernoulli random variable of parameter
$p$ (also
$\sigma_i$ and $\omega^i$ are independent). Now, let us simplify the notation
by setting $f_i (\sigma_i ) \stackrel{\mathrm{def}}{=}\widehat{\bbE
}(f (\omega)|\sigma_i
)$.
Then, using $\widehat{\bbP}(\sigma_i=1) = p$ in the second equality below,
%
\begin{eqnarray}
\label{eq:a} (1-p) \bigl(f_i (1)-f_i(0)\bigr) &
\stackrel{\mathrm{def}} {=}& (1-p) \biggl[\frac{\widehat{\bbE}(f(\omega)\1_{\sigma
_i=1})}{\widehat{\bbP}(\sigma_i=1)} -\frac{\widehat{\bbE}(f (\omega)\1_{\sigma_i =0})}{\widehat{\bbP
}(\sigma
_i=0)}
\biggr]
\nonumber
\\
&=&\biggl(\frac{1}p-1\biggr)\widehat{\bbE}\bigl(f (\omega)
\1_{\sigma_i
=1}\bigr)-\widehat{\bbE} \bigl(f (\omega)\1_{\sigma_i =0}\bigr)
\nonumber
\\
&\stackrel{\mathrm{(\mathsf{P1})}} {=}&\frac{1}p\widehat{\bbE}\bigl(f (
\omega )\1_{\omega_i=1}\1_{\sigma_i =1}\bigr)-\bbE\bigl(f (\omega)\bigr)
\\
&\stackrel{\scriptsize{\eqref{eq:defZ}}} {=} &\frac{1}p\bbE \biggl( f (\omega )
\1_{\omega_i=1}\frac{p}{\bbP(\omega_i =1|
\omega^i)} \biggr)-\bbE\bigl(f (\omega)\bigr)
\nonumber
\\
&=& \bbE\bigl(f \bigl(\omega^i \times1\bigr)\bigr)-\bbE
\bigl(f (\omega )\bigr)\stackrel{\mathrm{(\mathsf{FKG})}} {\ge}0.\nonumber
\end{eqnarray}

\emph{Property $\mathsf{P3}$.} We wish to prove the negative
association formula \eqref{eq:NegAs}. Since $g$ is a nondecreasing
function of
only one site, we only need to treat the case $g=\mathrm{id}$ (any
nondecreasing
function of one site is of the form $\alpha \mathrm{id}+\beta$ with
$\alpha\ge0$). For $\omega_I\in\{0,1\}^I$ and
$\omega^I\in\{0,1\}^{\Lambda\setminus I}$, let $\omega_I\times
\omega^I$ be the
configuration in $\Omega$ coinciding with $\omega_I$ on $I$ and
$\omega^I$ on
$\Lambda\setminus I$.
\begin{claim*}
For any subset $I\subset\Lambda$, the following happens:
If $\omega_I\succ\widetilde\omega_I$, then, for
any $\omega^I$ and for
any nondecreasing function
$f:\{0,1\}^{\Lambda\setminus I}\rightarrow\bbR$,
%
\begin{equation}
\label{eq:StDomZ} \widehat{\bbE}\bigl(f\bigl( \sigma^I\bigr) |
\omega_I\times\omega^I \bigr)\leq \widehat{\bbE}\bigl(f
\bigl( \sigma^I \bigr) |\widetilde\omega_I\times
\omega^I \bigr).
\end{equation}
\end{claim*}
\begin{pf}
Under the conditional measure $\widehat\bbP(\cdot|\omega
)$, the
sequence $\sigma$ is simply a collection of independent Bernoulli random
variables with probabilities of success specified by \eqref{eq:defZ}.
By \eqref{eq:strongFKG},
\[
\frac{p}{\bbP ( \omega_i = 1 | (\omega_I\times\omega
^{I})^i  )} \leq \frac{p}{\bbP ( \omega_i = 1 | (\widetilde\omega_I\times
\omega^{I})^i
 )},
\]
for any $i\notin I$ and $\omega_I\succ\widetilde\omega_I$, a
fact which
implies
that the random variable $\sigma_I$ conditioned on $\omega_I\times
\omega^{I}$
is stochastically dominated by the random variable $\sigma_I$
conditioned on
$\widetilde\omega_I\times\omega^I$. The claim follows by definition of
stochastic domination.
\end{pf}
In particular, the claim yields that if $f$
is a nondecreasing function of $\sigma^i$, then for any $i$ and any
$\omega^i$,
%
\begin{equation}
\label{eq:StDom-i} \widehat{\bbE}\bigl(f \bigl(\sigma^i\bigr)|
\omega^i\times1\bigr)\leq \widehat{\bbE}\bigl(f \bigl(
\sigma^i\bigr)|\omega^i\bigr).
\end{equation}
As a result, we infer
\begin{eqnarray*}
\widehat{\bbE}\bigl(f \bigl(\sigma^i\bigr)\sigma_i\bigr)
& = &\widehat{\bbE}\bigl(f\bigl(\sigma^i\bigr)\1_{\sigma_i=1}
\1_{\omega_i=1}\bigr)
\\
&\stackrel{\scriptsize{\eqref{eq:defZ}}} {=} &\sum_{\omega^i } p
\widehat{\bbE}\bigl(f \bigl(\sigma^i\bigr)|\omega^i\times1
\bigr) \bbP\bigl( \omega^i\bigr)
\\
&\stackrel{\scriptsize{\eqref{eq:StDom-i}}} {\leq}& p \sum_{\omega^i }
\widehat{\bbE}\bigl(f \bigl(\sigma^i\bigr)|\omega^i\bigr)
\bbP\bigl( \omega^i\bigr)
\\
& = &p \widehat{\bbE}\bigl[f \bigl(\sigma^i\bigr)\bigr]\stackrel{
\mathrm{(\mathsf {P1})}} {=}\widehat{\bbE} [\sigma_i] \widehat{\bbE}
\bigl[f \bigl(\sigma^i\bigr)\bigr].
\end{eqnarray*}

\emph{Property $\mathsf{P4}$.} Assume that
$X\stackrel{\mathrm{def}}{=}\sum_{i\in\Lambda}\lambda_i(\sigma
_i-p)=0$. We deduce that
\[
0 =\widehat\bbE\bigl[X^2\bigr] =\widehat\bbE\bigl[\bigl(X-\widehat
\bbE[X|\omega]\bigr)^2\bigr] +\widehat\bbE\bigl[\widehat\bbE[X|
\omega]^2\bigr] \ge\widehat\bbE\bigl[\bigl(X-\widehat\bbE[X|\omega]
\bigr)^2\bigr].
\]
Now,
\[
X-\widehat\bbE[X|\omega] =\sum_{i\in\Lambda}
\lambda_i\bigl(\sigma_i-\widehat\bbE[\sigma
_i|\omega]\bigr)
\]
and, conditionally on $\omega$, the random variables
$\sigma_i-\widehat\bbE[\sigma_i|\omega]$ are independent and have
mean 0. We deduce that
\begin{eqnarray*}
\widehat\bbE\bigl[\bigl(X-\widehat\bbE[X|\omega]\bigr)^2\bigr] &=&
\widehat\bbE \bigl[\widehat\bbE\bigl[\bigl(X-\widehat\bbE[X|\omega ]
\bigr)^2|\omega\bigr] \bigr]
\\
&=&\widehat\bbE \biggl[\widehat\bbE \biggl[ \biggl(\sum
_{i\in\Lambda} \lambda_i\bigl(\sigma_i-\widehat
\bbE[\sigma_i|\omega]\bigr) \biggr)^2\Big | \omega \biggr]
\biggr]
\\
&=&\widehat\bbE \biggl[\sum_{i\in\Lambda}
\lambda_i^2\widehat\bbE \bigl[ \bigl(\sigma_i-
\widehat\bbE[\sigma_i|\omega]\bigr)^2 |\omega \bigr]
\biggr]
\\
&=&\sum_{i\in\Lambda}\lambda_i^2
\widehat\bbE\bigl[\bigl(\sigma_i-\widehat \bbE[ \sigma_i|
\omega]\bigr)^2\bigr].
\end{eqnarray*}
Now, $\widehat\bbE[\sigma_i|\omega]\le1/2$ by \eqref{eq:defZ}
and \eqref{eq:FE} and, therefore,
%
\begin{eqnarray}
\label{eq:important} \widehat\bbE\bigl[\bigl(\sigma_i-\widehat\bbE[
\sigma_i|\omega]\bigr)^2\bigr] &=& \bbE \bigl[\widehat\bbE[
\sigma_i|\omega]\bigl(1-\widehat\bbE[\sigma _i|\omega]
\bigr) \bigr]
\nonumber
\\[-8pt]
\\[-8pt]
\nonumber
&\ge& \frac{1}2\bbE\bigl[\widehat\bbE[\sigma_i|
\omega]\bigr]=\frac{p}2>0,
\end{eqnarray}
which implies that $\lambda_i=0$ for every $i\in\Lambda$.
\end{pf*}

\subsection{Proof of Theorem \texorpdfstring{\protect\ref{thm:BK-Var}}{1.1}}
\label{sub:BK-Var}
Recall our notation $f_i (\sigma_i ) \stackrel{\mathrm
{def}}{=}\widehat{\bbE}(f (\omega)|
\sigma_i
)$.
We may represent $f (\omega)$ as
\[
f=\sum_{i}\gamma_i f_i (
\sigma_i)+f^\perp,
\]
with $f^\perp$ orthogonal to the subspace of
$\bbL^2(\widehat\Omega,\widehat\bbP)$ spanned by $(f_i(\sigma
_i),i\in\Lambda)$.
Since $\widehat\bbE[f_i(\sigma_i)]=0$ and
$\widehat\bbE[\sigma_i]=p$ by $\mathsf{P1}$ of Theorem~\ref
{thm:uZ}, we deduce
that
%
\begin{equation}
\label{eq:formula f_i} f_i(\sigma_i)=\bigl(f_i(1)-f_i(0)
\bigr) (\sigma_i-p).
\end{equation}
Should the random variables $\sigma_i$ be independent, we would immediately
infer that $\gamma_i\equiv1$ and
\[
\Var\bigl(f(\omega)\bigr) \geq p (1-p) \sum_i
\bigl(f_i(1)-f_i(0)\bigr)^2,
\]
which, by \eqref{eq:a}, would be the end of the proof. This is
precisely the
computation done in the Bernoulli case in \cite{ChS13}. In our case, however,
the random variables $\sigma_i$ are dependent, and we need additional
information and more care in order to control both the coefficients
$\gamma_i$ and the cross-terms. It is precisely at this stage that negative
dependence, as stated in $\mathsf{P3}$ of Theorem~\ref{thm:uZ}, becomes
crucial.

Before proceeding with the proof of the theorem, let us formulate and prove
the following elementary lemma.

\begin{lemma}
\label{lem:G-Fact}
Let $(\calH, \langle\cdot, \cdot\rangle)$ be a finite dimensional
Hilbert space, and let $\{\mathsf f_1,\ldots,\mathsf f_n\}$ be a
normalized basis of
$\calH$, such that $ \langle\mathsf f_i, \mathsf f_j \rangle\leq0$
for every pair
$i\neq j$. Then, for any $\mathsf f=\sum_{i=1}^n\lambda_i\mathsf f_i$
such that
$\langle\mathsf f,\mathsf f_i\rangle\ge0$ for every $1\le i\le n$,
we have that
$\lambda_i\ge\langle\mathsf f, \mathsf f_i \rangle$ for every $1\le
i\le n$.
\end{lemma}

Note that when the basis $\{\mathsf f_1,\ldots,\mathsf f_n\}$ is
orthogonal, we find that
$\lambda_i=\langle\mathsf f,\mathsf f_i\rangle$ for every $1\le i\le n$.
\begin{pf*}{Proof of Lemma \ref{lem:G-Fact}}
Lemma~\ref{lem:G-Fact} relies on the following transparent geometric
fact.

\textit{Obtuse cone property}.
Consider a positive
cone $\calC= \calC(\mathsf f_1, \ldots, \mathsf f_n )$
spanned by vectors $\mathsf f_i$. The cone is called obtuse if $
\langle\mathsf f_i,\mathsf f_j\rangle\le0$ for any $i\neq j$.
Then $\langle\mathsf f,\mathsf f_i\rangle\ge0$ for every $ i$
implies that $\mathsf f\in\calC$, that is, all the coefficients
$\lambda_j$ in the
decomposition $\mathsf f= \sum_j\lambda_j\mathsf f_j$ are
nonnegative. 

Taking scalar product with normalized vectors $\mathsf f_i$ yields
\[
\langle\mathsf f,\mathsf f_i\rangle= \lambda_i + \sum
_{j\neq i} \lambda_j \langle\mathsf
f_j,\mathsf f_i\rangle\leq\lambda_i
\]
and one then gets the conclusion of Lemma~\ref{lem:G-Fact}.

Perhaps the simplest way to prove the above obtuse cone property
is by induction on the cardinality of the basis. The two-dimensional
case is
straightforward.
Let us now consider the basis $\{\mathsf f_1,\ldots,\mathsf f_{n+1}\}$.
For $i\ge2$, let
%
\begin{equation}
\label{eq:Fi} \sfF_i \stackrel{\mathrm{def}} {=}\frac{\mathsf f_i-\langle\mathsf
f_i,\mathsf
f_1\rangle\mathsf f_1}{
1-\langle\mathsf f_1,\mathsf f_i\rangle^2}
\end{equation}
be the normalized projection of $\mathsf f_i$ on $\operatorname{vec}(\mathsf
f_1)^\perp$. Also,
write $\mathsf f=\langle\mathsf f,\mathsf f_1\rangle\mathsf f_1+\sfF
$ where
$\sfF\in\operatorname{vec}(\sfF_2,\ldots,\sfF_{n+1})$. We wish to apply
the induction
hypothesis with $\sfF$ and $\sfF_2,\ldots,\sfF_{n+1}$. For that,
simply observe
that, for $i\ne j$,
\[
\langle\sfF_i,\sfF_j\rangle = \frac{\langle\mathsf f_i,\mathsf f_j\rangle
-\langle\mathsf f_i,\mathsf f_1\rangle\langle\mathsf f_j,\mathsf
f_1\rangle}{
(1-\langle\mathsf f_1,\mathsf f_i\rangle^2)(1-\langle\mathsf
f_1,\mathsf f_j\rangle^2)}\le0
\]
(we used that $\langle\mathsf f_i,\mathsf f_j\rangle\le0$ for $i\ne
j$) and
%
\begin{equation}
\label{eq:easy} \langle \sfF,\sfF_i\rangle= \langle\mathsf f,
\sfF_i\rangle= \frac{\langle\mathsf f,\mathsf f_i\rangle-\langle\mathsf f_i,
\mathsf f_1\rangle\langle\mathsf f,\mathsf f_1\rangle}{1-\langle
\mathsf f_i,
\mathsf f_1\rangle^2}\ge\langle\mathsf f,\mathsf
f_i\rangle\ge0,
\end{equation}
where the first equality is due to the orthogonality of $\mathsf f_1$
and $\sfF_i$,
and the first inequality
to the fact that $\langle\mathsf f_i,\mathsf f_1\rangle\le0$ and
$\langle\mathsf f,\mathsf f_1\rangle\ge0$. Therefore, by the
induction assumption, $\sfF$
lies in the cone $\calC(\sfF_2, \ldots, \sfF_{n+1} )$. By \eqref
{eq:Fi}, all
$\sfF_i$'s lie in $\calC(\mathsf f_1, \ldots, \mathsf f_{n+1})$ and
hence $\sfF\in
\calC(\mathsf f_1, \ldots, \mathsf f_{n+1})$ as well. Since the coefficient
$\langle\mathsf f,\mathsf f_1\rangle$ in $\mathsf f=\langle\mathsf
f,\mathsf f_1\rangle\mathsf f_1+
\sfF$ is nonnegative, we are done.
\end{pf*}

\begin{pf*}{Proof of Theorem~\ref{thm:BK-Var}}
Consider a field $\omega\in\{0,1\}^\Lambda$ satisfying ($\mathsf
{FE}$) and
($\mathsf{FKG}$). We keep\vspace*{1pt} the notation from the previous section. In
particular, on the probability space $(\widehat\Omega,\widehat\bbP
)$, we
associate the field $\sigma$ to $\omega$.

Let $f$ be a square-integrable nondecreasing function of $\omega$. As before,
we set $f_i(\sigma_i)\stackrel{\mathrm{def}}{=}\widehat\bbE
[f(\omega)|\sigma_i]$.
Without loss of
generality, we assume that $\bbE[f(\omega)]=0$, and consequently,
$\widehat\bbE[f_i(\sigma_i)]=0$ for every $i\in\Lambda$.

Let $I$ be the set of indices $i$ satisfying $\bbE[f_i(\sigma_i)^2]>0$.
Consider $V=\operatorname{vec}(f_i(\sigma_i):i\in I)$ and write
\[
f = \sum_{i\in I}\gamma_if_i(
\sigma_i)+f^\perp,
\]
where $f^\perp\in V^\perp$. Properties $\mathsf{P2}$ and $\mathsf
{P3}$ of
Theorem~\ref{thm:uZ} show that $\langle f_i(\sigma_i), f_j(\sigma
_j)\rangle\le0$
for every $i\ne j$ in $I$.
Furthermore,
\[
\bigl\langle f(\omega),f_i(\sigma_i)\bigr\rangle =
\widehat\bbE\bigl[f(\omega)f_i(\sigma_i)\bigr] = \widehat
\bbE\bigl[f_i(\sigma_i)^2\bigr] \ge0.
\]
(The last equality is due to the definition of the conditional expectation.)
Last but not least, the family $\{f_i(\sigma_i):i\in I\}$ forms a
basis of
$V$. Indeed, this directly follows from \eqref{eq:formula f_i}.

Now, $f_i(1)-f_i(0)>0$ for $i\in I$, and the family $\{\sigma_i-p:i\in
I\}$ is
free thanks to $\mathsf{P4}$. We are therefore in position to apply the
previous lemma with $\mathsf f_i(\sigma_i)\stackrel{\mathrm
{def}}{=}f_i(\sigma_i)/\|
f_i(\sigma_i)\|$ in
order to obtain that
%
\begin{equation}
\label{eq:al} \gamma_i \ge \frac{\langle f(\omega),f_i(\sigma_i)\rangle}{\|f_i(\sigma_i)\|^2} = \frac{\widehat\bbE[f_i(\sigma_i)^2]}{\|f_i(\sigma_i)\|^2} =
1.
\end{equation}
We deduce that
\begin{eqnarray*}
\widehat\bbE\bigl[f(\omega)^2\bigr] &\ge& \widehat\bbE \biggl[
\biggl(\sum_{i\in I}\gamma_i f_i(
\sigma_i) \biggr)^2 \biggr]
\\
&= &\widehat\bbE \biggl[ \biggl(\sum_{i\in I}
\gamma_i\bigl(f_i(\sigma_i) -\widehat\bbE
\bigl[f_i(\sigma_i)|\omega\bigr]\bigr)
\biggr)^2 \biggr] +\widehat\bbE \biggl[ \biggl(\sum
_{i\in I}\gamma_i\widehat\bbE \bigl[f_i(
\sigma_i)|\omega\bigr] \biggr)^2 \biggr]
\\
&\ge& \widehat\bbE \biggl[ \biggl(\sum_{i\in I}
\gamma_i \bigl(f_i(\sigma_i)-\widehat\bbE
\bigl[f_i(\sigma_i)|\omega\bigr]\bigr)
\biggr)^2 \biggr]
\\
&=& \sum_{i\in I} \gamma_i^2
\widehat\bbE\bigl[\bigl(f_i(\sigma_i) -\widehat\bbE
\bigl[f_i(\sigma_i)|\omega\bigr]\bigr)^2\bigr]
\\
& \stackrel{\scriptsize{\eqref{eq:al}}} {\ge}& \sum_{i\in I} \widehat
\bbE\bigl[\bigl(f_i(\sigma_i) -\widehat\bbE
\bigl[f_i(\sigma_i)|\omega\bigr]\bigr)^2\bigr]
\\
&=& \sum_{i\in\Lambda} \widehat\bbE\bigl[
\bigl(f_i(\sigma_i)-\widehat\bbE\bigl[f_i(
\sigma_i)|\omega\bigr] \bigr)^2\bigr],
\end{eqnarray*}
where the second equality is due to the fact that, conditionally on
$\omega$,
the random variables
$ \{f_i(\sigma_i)-\widehat\bbE[f_i(\sigma_i)|\omega]:i\in I
 \}$ are
orthogonal (since the $\sigma_i-\widehat\bbE[\sigma_i|\omega]$ are),
and the last equality to the observation that, for $i\notin I$,
\[
0 \le \widehat\bbE\bigl[\bigl(f_i(\sigma_i)-\widehat\bbE
\bigl[f_i(\sigma_i)|\omega\bigr]\bigr)^2\bigr]
\le \widehat\bbE\bigl[f_i(\sigma_i)^2
\bigr]=0.
\]
We are now ready to conclude. Similarly to \eqref{eq:formula f_i}, we
find that (remember that we chose $p=\cFE/2$)
%
\begin{eqnarray}\label{eq:BK-Var-Xi}
&&\widehat\bbE \bigl[\bigl(f_i(\sigma_i)-\widehat\bbE
\bigl[f_i(\sigma _i)|\omega\bigr]\bigr)^2
\bigr] \nonumber\\
&&\qquad= \bigl(f_i(1)-f_i(0)\bigr)^2\widehat
\bbE \bigl[\bigl(\sigma_i-\widehat\bbE[\sigma _i|\omega]
\bigr)^2 \bigr]
\\
&&\qquad \stackrel{\scriptsize{\eqref{eq:important}}} {\ge}\frac{p}2\bigl(f_i(1)-f_i(0)
\bigr)^2
\nonumber
\\
&&\qquad \stackrel{\scriptsize{\eqref{eq:a}}} {=}\frac{p}{2(1-p)^2} \bigl(\bbE \bigl[f\bigl(
\omega^i\times1\bigr)\bigr] -\bbE\bigl[f(\omega)\bigr]
\bigr)^2\nonumber
\\
&&\qquad=\frac{p}{2(1-p)^2} \bigl(\bbE\bigl[\bigl(f\bigl(\omega^i\times1
\bigr)-f\bigl(\omega ^i\times0\bigr)\bigr)\1_{
\omega_i=0}\bigr]
\bigr)^2
\nonumber
\\
&&\qquad\stackrel{\scriptsize{\eqref{eq:FE}}} {\ge} \frac{2p^3}{(1-p)^2} \bigl(\bbE\bigl[ f\bigl(
\omega^i\times1\bigr)-f\bigl(\omega^i\times0\bigr)\bigr]
\bigr)^2 .\nonumber
\end{eqnarray}
Overall, we find that
\[
\Var\bigl(f(\omega)\bigr)=\widehat\bbE\bigl[f(\omega)^2\bigr] \ge
\frac{2p^3}{(1-p)^2}\sum_{i\in
\Lambda} \bigl(\bbE\bigl[f\bigl(
\omega^i\times1\bigr) - f\bigl(\omega^i\times0\bigr)\bigr]
\bigr)^2.
\]
\upqed\end{pf*}

\begin{remark}
\label{rem:LB-FE}
Observe that, up to \eqref{eq:BK-Var-Xi}, the proof only made use of
the lower
bound $\bbP(\omega_i = 1 |\omega^i )\geq\cFE=2p$ in \eqref{eq:FE}.
\end{remark}

\section{Applications}

\subsection{Proof of Proposition \texorpdfstring{\protect\ref{prop:4 arm torus}}{2.1}}
\mbox{}
\begin{pf}
Let $E_n$ be the event that there exists an open circuit with nontrivial
homotopy in $\bbT_n$. Theorem~\ref{thm:BK-Var} implies that
\[
\sum_{i\in E_{\bbT_n}}\bbP\bigl[\operatorname{Piv}_i(E_n)
\bigr]^2 \le \frac{1}{\cP}\bbP(E_n) \bigl(1-
\bbP(E_n)\bigr)\le\frac{1}{\cP}.
\]
Note that edges are of two types: either ``vertical'' or ``horizontal''.
By shift invariance of $\bbP$ and of the event $E_n$ we therefore
obtain, for
an horizontal edge $e$ (the same reasoning can be applied to vertical edges),
\[
|\bbT_n|\cdot\bbP\bigl[\tilde A_2^e(n)
\bigr]^2 \le \sum_{i\in E_{\bbT_n}}\bbP\bigl[
\operatorname{Piv}_i(E_n)\bigr]^2\le
\frac{1}{\cP}.
\]
\upqed\end{pf}

\subsection{Quantitative Burton--Keane argument}
Recall that we are working with (nearest neighbor) bond percolation models.
For $x\in\Lambda_n$, the set $\calC_n (x)$ is the connected
component of $x$
in the restriction of the percolation configuration to the edges with
at least
one end-point in $\Lambda_n$.

For $x\in\Lambda_n$, let $\operatorname{Trif}_{n}(x)$ be the event that:

\begin{longlist}[(a)]
\item[(a)] There are \emph{exactly} three open bonds incident to $x$.
\item[(b)]$\calC_n (x )\setminus x$ is a disjoint union of \emph{exactly} three
connected clusters, and each of these three clusters is connected to
$\partial\Lambda_{n+1} \stackrel{\mathrm{def}}{=}\Lambda
_{n+1}\setminus\Lambda_n$.
\end{longlist}

Recall the following classical fact \cite{BurKea89}:

\begin{lemma}
\label{lem:TrifBound}
Consider a percolation model on $\Lambda_{n}$. Then, for any $\omega$,
%
\begin{equation}
\label{eq:TrifBound} \sum_{x\in\Lambda_n} \1_{\operatorname{Trif}_{n}(x)} (\omega)
\le|\partial\Lambda_{n+1}|.
\end{equation}
\end{lemma}

Note that this lemma has the following useful consequence: if $\bbP$
is in fact
a translation invariant measure on the whole plane or on a torus, it implies
that
%
\begin{equation}
\label{trif} \bbP\bigl[\operatorname{Trif}_{2n}(0)\bigr]\le
\frac{ |\partial\Lambda
_{n+1}|}{|\Lambda_n|}.
\end{equation}
Let $k\le n$. Define the event $\operatorname{CoarseTrif}_{k,n}$ that there
are at
least
three distinct clusters in the annulus $A_{k,n}\stackrel{\mathrm
{def}}{=}\Lambda
_n\setminus
\Lambda_k$
connecting the inner to the outer boundaries of $A_{k,n}$. In other words,
there are at least three distinct crossing clusters of $A_{k,n}$.

\begin{corollary}\label{cor:coarse trifurcation}
Consider a percolation model on $\bbZ^d$ satisfying $(\mathsf{FE})$,
$(\mathsf{FKG})$ and $(\mathsf{TI})$. There exists $c_1=c_1(\cFE)>0$
such that, for any $0\le k\le n$,
\[
\bbP[\operatorname{CoarseTrif}_{k,n}]\le\frac{\exp(c_1k)}{n}.
\]
\end{corollary}

\begin{pf}
By conditioning on the clusters in $A_{k,n}$, one may easily check that
\[
\bbP\bigl[\operatorname{Trif}_{n}(0)|\operatorname{CoarseTrif}_{k,n}
\bigr] \ge {\cFE}^{6d k}.
\]
(There may be some problems if the three clusters reach $\partial
\Lambda_k$
near the corner, yet such cases can be treated separately.)
The result follows readily from Lemma~\ref{lem:TrifBound}.
\end{pf}

\begin{pf*}{Proof of Theorem~\ref{prop:quantitative BK}}
Set $\ep\stackrel{\mathrm{def}}{=}1/(2c_1)$ and let $k=\lfloor\ep
\log n\rfloor$. Let
$E_n$ be
the
event that there are exactly two clusters in $A_{k,n}$ from the inner
to the
outer boundaries. On $E_n$, let $\calC$ be the set of vertices of $A_{k,n}$
connected to the boundary of $\Lambda_n$ by an open path. Since there
are only
two distinct clusters connecting the inside and outside boundaries of
$A_{k,n}$,
the vertices of $\calC\cap\Lambda_k$ can be divided into two subsets
$E_1\stackrel{\mathrm{def}}{=}E_1(\calC)$ and $E_2\stackrel{\mathrm
{def}}{=}E_2(\calC)$ depending on which clusters
they belong
to.

For every possible realization $C$ of $\calC$ so that $\{\calC=C\}
\subset E_n$,
define $\operatorname{Cross}(C)$ to be the event that $E_1$ and $E_2$ are connected
by
an
open path inside $\Lambda_k$.
Theorem~\ref{thm:BK-Var} applied to $\bbP(\cdot|\calC=C)$ and
$\operatorname{Cross}(C)$
gives
\[
\sum_{e\in E_{\Lambda_k}}\bbP\bigl[\operatorname{Piv}_e
\bigl(\operatorname {Cross}(C)\bigr)|\calC=C\bigr]^2\le \frac{1}4,
\]
which, by Cauchy--Schwarz, implies that
\[
\sum_{e\in E_{\Lambda_k}}\bbP\bigl[\operatorname{Piv}_e
\bigl(\operatorname {Cross}(C)\bigr)|\calC=C\bigr]\le \frac{1}2
\sqrt{|E_{\Lambda_k}|}.
\]
For an edge $e'$, let $A_2(n,e')$ be the event that the two end-points
of $e'$
are connected to the boundary of $\Lambda_n$ by two disjoint clusters. By
definition of $C$, we have that
$\operatorname{Piv}_{e'}(\operatorname{Cross}(C))\cap\{\calC=C\}=A_2(n,e')\cap
\{\calC=C\}$
which, by summing over
all possible $C$, implies that
\[
\sum_{e'\in E_{\Lambda_k}}\bbP\bigl[A_2
\bigl(n,e'\bigr),E_n\bigr]\le \frac{1}2
\sqrt{|E_{\Lambda_k}|} \bbP(E_n)\le \frac{1}2
\sqrt{|E_{\Lambda_k}|}\le\sqrt{d|\Lambda_k|/2}.
\]
It only remains to see that Corollary~\ref{cor:coarse trifurcation} implies
\[
\bbP\bigl[A_2\bigl(n,e'\bigr),\operatorname{CoarseTrif}_{k,n}
\bigr]\le \bbP[\operatorname{CoarseTrif}_{k,n}]\leq\frac{\exp(c_1\ep\log
n)}{n}=
\frac{1}{\sqrt n}.
\]
At the end, we find that
\[
\frac{1}{|\Lambda_k|}\sum_{e'\in E_{\Lambda_k}}\bbP
\bigl[A_2\bigl(n,e'\bigr)\bigr]\le \frac{\sqrt{d/2}}{\sqrt{|\Lambda_k|}}+
\frac{1}{\sqrt n}.
\]
Consider for a moment that the edge $e$ involved in $A_2^e(2n)$ is horizontal.
Since $A_2^e(2n)$ is included in a translate of $A_2(n,e')$ for any
horizontal edge $e'\in E_{\Lambda_k}$, we conclude that
\[
\bbP\bigl[A_2^e(2n)\bigr]\le\frac{1}{|\Lambda_k|}\sum
_{e'\in
E_{\Lambda_k}}\bbP\bigl[A_2\bigl(n,e'
\bigr)\bigr]\le \frac{\sqrt{d/2}}{\sqrt{|\Lambda_k|}}+\frac{1}{\sqrt n},
\]
which proves the theorem with $\cBK=\cBK(\cFE)>0$ small enough
thanks to
our choice for $k$.
\end{pf*}

\subsection{Continuity of percolation probabilities away from critical points}
Theorem~\ref{thm:cont} is an easy consequence of the quantitative Burton--Keane
bound
\eqref{eq:QBK}. Let $\alpha_0\in(a,b)$. Pick $\varepsilon>0$ such that
$[\alpha_0-\varepsilon, \alpha_0 +\varepsilon]$ is still
in $(a,b)$. By our assumptions on the family $\{\bbP_\alpha\}$,
%
\begin{equation}
\label{eq:Lambda-box-n} \limsup_{n\to\infty} \sup_{\alpha\in[\alpha_0-\varepsilon, \alpha
_0 +\varepsilon]}
\bbP_\alpha(\partial\Lambda_n \nleftrightarrow\infty) =
\limsup_{n\to\infty}\bbP_{\alpha_0 -\varepsilon} (\partial
\Lambda_n \nleftrightarrow\infty)= 0.
\end{equation}
On the other hand, for any $N >n$ and any $\alpha\in[\alpha
_0-\varepsilon,
\alpha_0 +\varepsilon]$,
%
\begin{eqnarray}
\label{eq:Split-Pa} \theta(\alpha) &\stackrel{\operatorname{def}} {=}&\bbP_\alpha (0
\leftrightarrow\infty)\nonumber\\
&=& \bbP_\alpha(0\leftrightarrow\partial
\Lambda_N)
\\
&&{}- \bbP_\alpha(0\leftrightarrow\partial\Lambda_N; 0
\nleftrightarrow\infty; \partial\Lambda_n \leftrightarrow \infty)-
\bbP_\alpha(0\leftrightarrow\partial\Lambda_N; \partial
\Lambda_n \nleftrightarrow\infty).\nonumber
\end{eqnarray}
The event $\{0\leftrightarrow\partial\Lambda_N;
0\nleftrightarrow\infty; \partial\Lambda_n \leftrightarrow
\infty\}$
implies the existence of at least two disjoint crossings of the annulus $A_{n,
N}$. If we choose $N = \mathrm{e}^{C n}$ for some sufficiently large
constant $C$, then the second term in \eqref{eq:Split-Pa} tends to
zero as
$n\to\infty$, uniformly in $\alpha\in[\alpha_0-\varepsilon, \alpha_0
+\varepsilon]$; this follows from \eqref{eq:QBK}, which in view of the assumed
uniformity of $\mathsf{(FE)}$ on compact sub-intervals, yields uniform upper
bounds for $\alpha\in[\alpha_0-\varepsilon, \alpha_0 +\varepsilon]$.
The third
term in \eqref{eq:Split-Pa} is controlled by \eqref{eq:Lambda-box-n}.
Hence, since the events $\{0\leftrightarrow\partial\Lambda_N\}
$ are
local, continuity of $\bbP_\alpha$ at $\alpha_0$ implies
that $\theta$ is continuous at $\alpha_0$ as well. \qed

\subsection{Proof of Theorem \texorpdfstring{\protect\ref{prop:4 arm}}{2.4}}\label{sec:RSW}
Let us recall some additional facts on the random-cluster model. First, let
us introduce random-cluster measures with boundary conditions. Fix a finite
graph $G$. \emph{Boundary conditions} $\xi$ are given by a partition
$P_1\sqcup
\cdots\sqcup P_k$ of $\partial G$. Two vertices are \emph{wired in
$\xi$} if
they
belong to the same $P_i$. The graph obtained from the configuration
$\omega$ by
identifying the wired vertices together is denoted by $\omega^\xi$. Let
$k(\omega^\xi)$ be the number of connected components of the graph
$\omega^\xi$.
The probability measure $\bbP^{\xi}_{p,q,G}$ of the random-cluster
model on $G$
with \emph{edge-weight} $p\in[0,1]$, \emph{cluster-weight} $q>0$ and boundary
conditions $\xi$ is defined by
%
\begin{equation}
\label{probconf} \bbP_{p,q,G}^{\xi} \bigl( \{\omega \}\bigr)
\stackrel{\mathrm{def}} {=} \frac{p^{o(\omega)}(1-p)^{c(\omega)}q^{k(\omega^\xi)}}{
Z_{p,q,G}^{\xi}}
\end{equation}
for every configuration $\omega$ on $G$. The constant $Z_{p,q,G}^{\xi
}$ is a
normalizing constant, referred to as the \emph{partition function},
defined in
such a way that the sum over all configurations equals 1. For $q\ge1$,
infinite-volume random-cluster measures can be defined as weak limits of
random-cluster measures on larger and larger boxes.

Recall that the planar random-cluster model possesses a dual model on
the dual
graph $(\bbZ^2)^\star$. The
configuration $\omega^\star\in\{0,1\}^{E_{(\bbZ^2)^\star}}$ is
defined as
follows: each dual-edge $e^\star\in(\bbZ^2)^\star$ is dual-open in
$\omega^\star$ if and only if the edge of $\bbZ^2$ passing through
its middle
(there is a unique such edge) is closed in $\omega$. If the law of
$\omega$ is
$\bbP_{p,q,G}^{\xi}$, then the law of the dual model is
$\bbP_{p^\star,q,G^\star}^{\xi^\star}$ for some dual boundary conditions
$\xi^\star$. We will only use that free and wired boundary conditions
are dual
to each other.

On $\bbZ^2$, the random-cluster model undergoes a phase transition at some
parameter $p_c(q)$ satisfying, for every infinite-volume random-cluster model
$\bbP_{p,q,\bbZ^2}$ with parameters $p$ and $q$,
\[
\bbP_{p,q,\bbZ^2}[0\longleftrightarrow\infty] = \cases{ \theta(p,q)>0, &\quad $
\mbox{if $p>p_c(q)$,}$\vspace*{2pt}
\cr
0, & \quad$\mbox{if
$p<p_c(q)$.}$}
\]
The critical point of the planar random-cluster model on $\bbZ^2$ is
known to
correspond to the self-dual point of the model, that is,
$p_c(q)\stackrel{\mathrm{def}}{=}\sqrt q/(1+\sqrt q)$ \cite
{BefDum12}. Also, for $q\in
[1,4]$, the
behavior at criticality is
known
to be the following (see \cite{DumSidTas13}): there is a unique
infinite-volume measure and, for any numbers $1<a < b \leq\infty$, there
exists $\cRSW= \cRSW(a,b)>0$ such that, for all $n\geq1$ and any boundary
conditions~$\xi$,
%
\begin{equation}
\label{eq:RSW} \cRSW\le\bbP_{p_c,q,\widehat R_n}^\xi \bigl[
\operatorname{Cross}(R_n) \bigr]\leq1-\cRSW,
\end{equation}
where
$R_n = R_n [a] \stackrel{\mathrm{def}}{=}[-an,an]\times[-n,n]$,
$\widehat R_n = \widehat
R_n [b]
\stackrel{\mathrm{def}}{=}[-bn,bn]\times[-2n,2n]$ and $\mathrm
{Cross}(R_n)$ is the event
that the
left-hand
and right-hand sides of $R_n$ are connected by an open path in $R_n$.

\begin{pf*}{Proof of Theorem~\ref{prop:4 arm}}
Fix $q\in[1,4]$. Note that, for $p< p_c(q)$, there exists $c=c(p,q)>0$ such
that, for any edge $e$ and every $n\ge1$,
\[
\bbP_{p,q,\bbZ^2}\bigl[A_2^e(n)\bigr] \le
\bbP_{p,q,\bbZ^2}[0\longleftrightarrow\partial\Lambda_n] \le
e^{-c(p,q)n},
\]
thanks to exponential decay of correlations, see \cite{BefDum12} one more
time.
Similarly, when $p>p_c(q)$, for any edge $e$ and every $n\ge1$,
\[
\bbP_{p,q,\bbZ^2}\bigl[A_2^e(n)\bigr] =
\bbP_{p^\star,q,(\bbZ^2)^\star}\bigl[A_2(n)\bigr] \le \bbP_{p^\star,q,(\bbZ^2)^\star}\bigl[u
\stackrel{\star } {\longleftrightarrow} \partial\Lambda_n^\star
\bigr]\le e^{-c(p^\star,q)n}.
\]
The only interesting case is therefore the critical point $p=p_c$.

\subsubsection{Proof using mixing properties and Proposition \texorpdfstring{\protect\ref{prop:4 arm
torus}}{2.1}}

Recall that, by Proposition~\ref{prop:4 arm torus}, we already know
that, for
any edge $e'$ of $\bbT_n^{(2)}$ and every $n\geq1$,
\[
\bbP_{p_c,q,\bbT_n^{(2)}}\bigl[\tilde A_2^{e'}(n)\bigr] \le
\frac{c_{\tilde{\mathsf{A}}_2}}{n},
\]
where $\bbP_{p_c,q,\bbT_n^{(2)}}$ is the random-cluster measure on
$\bbT_n^{(2)}$. We therefore only need to show that there exists $C>0$
such that
\[
\bbP_{p_c,q,\bbZ^2}\bigl[A_2^e(n)\bigr]\le C
\bbP_{p_c,q,\bbT
_n^{(2)}}\bigl[\tilde A_2^{e'}(n)\bigr].
\]
Embed $\bbT_n^{(2)}$ into $\bbZ^2$ in such a way that the vertex set is
$\Lambda_n\stackrel{\mathrm{def}}{=}[-n,n]^2$ and $e$ is an edge
having 0 as an endpoint.
First, we
wish to highlight that \eqref{eq:RSW} (more precisely the mixing
result \cite{Dum13}, Theorem~5.45) classically implies the
existence of $c_1>0$ such that, for every boundary conditions $\xi$
and $n\ge1$,
\[
\bbP_{p_c,q,\bbZ^2}\bigl[A_2^e(n/2)\bigr]\le
c_1\bbP_{p_c,q,\Lambda_n}^\xi \bigl[A_2^e(n/2)
\bigr].
\]
In particular, this is also true for so-called \emph{periodic boundary
conditions}, so that
%
\begin{equation}
\label{eq:11} \bbP_{p_c,q,\bbZ^2}\bigl[A_2^e(n/2)\bigr]
\le c_1\bbP_{p_c,q,\bbT_n^{(2)}}\bigl[A_4^e(n/2)
\bigr].
\end{equation}
Now, introduce the event $A_2^\mathrm{sep}(n/2)$ that there exist two open
paths
$\gamma,\tilde\gamma$ and two dual-open dual-paths $\gamma^\star$ and
$\tilde\gamma^\star$, originating from the endpoints of $e$ and $e^*$,
respectively, satisfying:
\begin{itemize}
\item the endpoints (on the boundary of $\Lambda_{n/2}$ and
$\Lambda_{n/2}^\star$, resp.) $x$, $\tilde x$, $x^\star$ and
$\tilde
x^\star$ of the paths are at distance larger or equal to $\frac
{n}{10}$ from
each others.
\item$x$ and $\tilde x$ are connected to $\partial\Lambda_{
{3n}/{5}}$ in
$x+\Lambda_{{n}/{10}}$ and $\tilde x+\Lambda_{{n}/{10}}$.
\item$x^\star$ and $\tilde x^\star$ are connected to
$\partial\Lambda_{{3n}/{5}}^\star$ in $x^\star+\Lambda
_{{n}/{10}}$ and
$\tilde x^\star+\Lambda_{{n}/{10}}$.
\end{itemize}
Classically, \eqref{eq:RSW} implies that there exists $c_2>0$ such
that, for
any $n\ge1$,
%
\begin{equation}
\label{eq:12} \bbP_{p_c,q,\bbT_n^{(2)}}\bigl[A_2^e(n/2)\bigr]
\le c_2 \bbP_{p_c,q,\bbT_n^{(2)}}\bigl[A_2^\mathrm{sep}(n/2)
\bigr].
\end{equation}
See \cite{Nol08} for a treatment in the case of Bernoulli percolation
and \cite{Dum13}, Theorems~10.22 and 10.23, for the FK--Ising model
(the proofs of
the theorems apply {mutatis mutandis} to any random-cluster
model with
$1\le q\le4$).

It remains to see that there exists $c_3>0$ such that, for any $n\ge1$,
\[
\bbP_{p_c,q,\bbT_n^{(2)}}\bigl[A_2^\mathrm{sep}(n/2)\bigr] \le
c_3 \bbP_{p_c,q,\bbT_n^{(2)}}\bigl[\tilde A_2^e(n)
\bigr].
\]
In order to do so, mimic the classical argument to prove quasi-multiplicativity
of arm-probabilities for Bernoulli percolation (see \cite{Nol08} again and
Figure~\ref{fig:1}).

\begin{figure}

\includegraphics{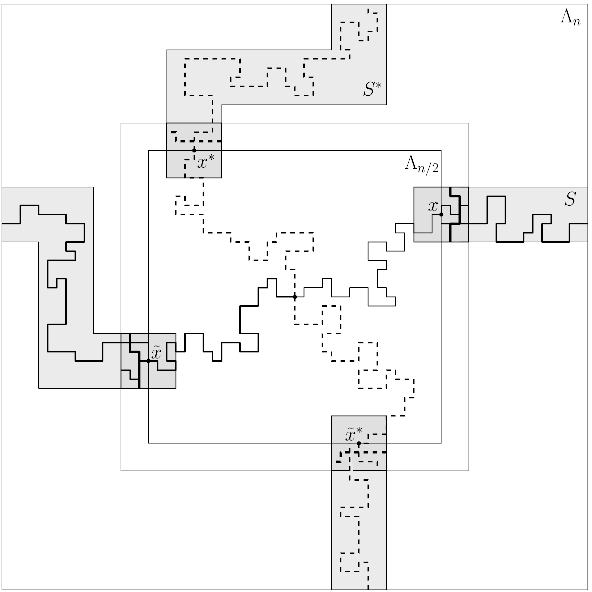}

\caption{The event $A_2^\mathrm{sep}(n/2)$ together with the
extension of the
four
paths using the sets $S$ and $S^*$. Estimates on crossing probabilities
available from \cite{DumSidTas13} show that these extensions cost a
multiplicative constant (not depending on $n$).}
\label{fig:1}
\end{figure}

We only sketch the proof.
Condition on $A_2^\mathrm{sep}(n/2)$. Consider a thin area $S$ of
width $\frac{n}{10}$ going from $x+\Lambda_{{n}/{10}}$ to $\tilde x+\Lambda
_{{n}/{10}}$ outside $\Lambda_{{n}/2}$, an a thin dual area $S^\star$
of width
$\frac{n}{10}$ going from $x^\star+\Lambda_{{n}/{10}}$ to $\tilde
x^\star+\Lambda_{{n}/{10}}$ outside $\Lambda_{{n}/2}$ so
that these two
areas do not intersect. Now, \eqref{eq:RSW} implies that there exist, with
probability $c_4>0$, a primal path in $S$ connecting the two paths
$\gamma$ and
$\tilde\gamma$, and a dual path in $S^\star$ connecting $\gamma
^\star$ and
$\tilde\gamma^\star$. But whenever this occurs, the event $\tilde
A_2^e(n)$ is
satisfied, so that
\[
\bbP_{p_c,q,\bbT_n^{(2)}}\bigl[\tilde A_4^e(n)\bigr] \ge
c_4 \bbP_{p_c,q,\bbT_n^{(2)}}\bigl[A_2^\mathrm{sep}(n/2)
\bigr].
\]
It only remains to invoke \eqref{eq:11} and \eqref{eq:12} to conclude.
\end{pf*}

\subsubsection{Proof using bounds on \texorpdfstring{$\bbE_{p_c,q,\bbZ^2} (N_{m,n
}^2)$}{E{pc,q,Z2}(N{m,n}2)}
and \texorpdfstring{\protect\eqref{eq:N-bound-cor}}{(2.4)}}
We shall check that there exists $\cN= \cN(\cRSW)<\infty$,
such that, uniformly in $m$,
%
\begin{equation}
\label{eq:Nsqr} \bbE_{p_c,q,\bbZ^2} \bigl(N_{m,5m }^2\bigr)
\leq\cN.
\end{equation}
A substitution into \eqref{eq:N-bound-cor} yields the claim.

Consider the annulus $A_{m, 5m}$ and the four rectangles
%
\begin{eqnarray}\qquad
\label{eq:strips} S_{m, \mathsf U}& =& [-5m, 5m]\times[m, 5m],\qquad S_{m, \mathsf R} =
[-5m, -m]\times[-5m, 5m],
\nonumber
\\[-8pt]
\\[-8pt]
\nonumber
S_{m, \mathsf L}& =& [m, 5m]\times[-5m, 5m],\qquad S_{m, \mathsf D} = [-5m, 5m]
\times[-5m, -m].
\end{eqnarray}
For $*\in\{{\mathsf U}, {\mathsf R}, {\mathsf L},{\mathsf D}\}$, let $N_{m,*} $ be the number
of distinct short-side crossing clusters of $S_{m *}$. For instance
$N_{m,\mathsf
U}$ is the
number of distinct clusters which connect $[-5m, 5m]\times\{m\}$ to
$[-5m, 5m]\times\{5m\}$ in the restriction of the percolation
configuration to the rectangle $S_{m, \mathsf U}$. Clearly,
\[
N_{m, 5m} \leq\sum_{*\in\{{\mathsf U}, {\mathsf R}, {\mathsf L},{\mathsf D}\}} N_{m,*},
\]
and, by symmetry, it remains to give an upper bound on
$\bbE_{p_c,q,\bbZ^2} (N_{m, {\mathsf U} }^2)$.

\begin{lemma}
\label{lem:Cl-size}
The RSW bound \eqref{eq:RSW} implies
\begin{equation}
\label{eq:Cl-size} \bbP_{p_c,q,\bbZ^2} (N_{m, {\mathsf U} } \geq k )\leq \bigl(1-
\cRSW(5,\infty)\bigr)^{k-1},
\end{equation}
uniformly in $k>1$ and $m$.
\end{lemma}

\begin{pf}
Let us introduce the events $\calR_k \stackrel{\mathrm{def}}{=}\{
N_{m, {\mathsf U} } \geq
k\}$.
We claim that
%
\begin{equation}
\label{eq:cond-k} \bbP_{p_c,q,\bbZ^2} (\calR_{k} |
\calR_{k-1})\leq 1- \cRSW(5,\infty),
\end{equation}
uniformly in $m$ and $k > 1$. Indeed, distinct crossing clusters which
show up
in any percolation configuration from $\calR_{k-1}$ are naturally
ordered from
left
to right. There are at least $(k-1)$ such clusters. The following somewhat
standard construction, which we sketch below, is depicted on
Figure~\ref{fig:cross}.

\begin{figure}

\includegraphics{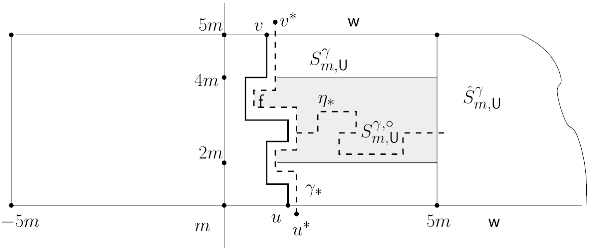}

\caption{The dual path $\gamma_* = \gamma_*^{u,v}$. Event $\calD
^\circ$: a dual
path $\eta_*$ crosses from left to right the middle section $S_{m,\mathsf
U}^{\gamma,\circ}$ of $S_{m,\mathsf U}^\gamma$ and, as such, rules out the
occurrence of the event $\calR_k$.
The boundary conditions (for the direct model) on the semi-infinite
strip $\hat
S_{m, \mathsf U}^\gamma$ are $\mathsf{w}$ on the upper and lower parts, and
$\mathsf{f}$ on~$\gamma_*$.}
\label{fig:cross}
\end{figure}

Consider the disjoint decomposition $\calR_{k-1} = \cup\calR_{k-1}^{u,v}$,
where $u$ (resp., $v$) is the rightmost vertex of the $(k-1)$th crossing
cluster on the bottom (resp., top) side of $S_{m, \mathsf U}$. The event
$\calR_{k-1}^{u,v}$ implies that there is the left-most dual crossing
$\gamma^{u,v}_*$ from $u^*$ to $v^*$,\vadjust{\goodbreak} where $u^* \stackrel{\mathrm
{def}}{=}u + \frac
{1}{2} (1,-1)$
and $v^* \stackrel{\mathrm{def}}{=}v +\frac{1}{2} (1,1)$. Consider
the remaining part,
denoted by
$S_{m, \mathsf U}^\gamma$, of the rectangle
$S_{m, \mathsf U}$ to the right of $\gamma^{u,v}_*$. Let $S_{m, \mathsf
U}^{\gamma,
\circ}$ be the middle section of $S_{m,\mathsf U}^\gamma$, that is, $S_{m,
\mathsf
U}^{\gamma, \circ} \stackrel{\mathrm{def}}{=}S_{m, \mathsf U}^\gamma
\cap(
\bbZ\times[2m, 4m])$. Finally, consider the
infinite strip extension $\hat S_{m, \mathsf U}^\gamma$ to the right of
$S_{m, \mathsf
U}^\gamma$.

Let $\calD^{\circ}$ be the event that there is a left to right dual
crossing of
$S_{m, \mathsf U}^{\gamma, \circ}$. By the FKG property of the random-cluster
model,
\[
\bbP_{p_c,q,\bbZ^2} (\calR_{k} | \calR_{k-1})\leq1- \min
_{\gamma} \bbP_{p_c,q,\hat S_{m, \mathsf U}^\gamma}^{\mathsf{w}, \mathsf{f}} \bigl(
\calD^{\circ}\bigr),
\]
where the boundary conditions are \emph{direct} boundary conditions on the
semi-infinite strip $\hat S_{m, \mathsf U}^\gamma$: wired on upper and
lower parts
and free on $\gamma$. Note that the model is self-dual at criticality.
Hence, for any possible realization of $\gamma$,
\[
\bbP_{p_c,q,\hat S_{m, \mathsf U}^\gamma}^{\mathsf{w}, \mathsf{f}} \bigl( \calD^{\circ}\bigr)\geq
\bbP_{p_c,q,\hat R_m [\infty]}^{\mathsf{f}} \bigl( \operatorname{Cross} \bigl(R_m
[5 ]\bigr)\bigr),
\]
and \eqref{eq:RSW} applies.
\end{pf}

\begin{remark}
Let us highlight the fact that SLE predictions,
see \cite{Dum13}, Section~13.3.2, suggest that
$\bbP_{p,q,\bbZ^2}[A_4^e(n)]=n^{-\xi_{1010}+o(1)}$, where
\[
\xi_{1010}\stackrel{\mathrm{def}} {=}\frac{3\sigma^2+10\sigma
+3}{4(1+\sigma)} \qquad\mbox{with }
\sigma\stackrel{\mathrm{def}} {=}\frac{2}\pi \arcsin(\sqrt q/2).
\]
This implies that $\bbP_{p,q,\bbZ^2}[A_4^e(n)]\gg\frac{1}n$ for
$q<2\arcsin[\pi\frac{2-\sqrt3}{\sqrt3}]\approx0.459$. This
illustrates the
fact that the claim of Theorem~\ref{thm:BK-Var} can fail to hold when the
condition $(\mathsf{FKG})$ is dropped.\vadjust{\goodbreak}
\end{remark}

\section*{Acknowledgment}
The authors thank Michael Aizenman for suggesting the use of the
number of
distinct clusters in Corollary~\ref{cor:N-bound-cor}.

%




\printaddresses
\end{document}